\documentclass[12pt,a4paper]{amsart}

\usepackage{fullpage}
\usepackage[pdftex]{epsfig}

\usepackage[OT2,T1]{fontenc}
\usepackage{amssymb}  
\usepackage{amsthm}
\usepackage{latexsym} 
\usepackage{comment}
\usepackage{color}
\usepackage{enumerate}
\usepackage[all]{xy}

\usepackage{charter,euler} 
\DeclareSymbolFont{bchoperators}{T1}{bch}{m}{n}
\SetSymbolFont{bchoperators}{bold}{T1}{bch}{b}{n}
\makeatletter
\renewcommand{\operator@font}{\mathgroup\symbchoperators}
\makeatother

\usepackage{titlesec} 
\titleformat{\section}{\normalfont\bfseries\filcenter}{\thesection}{1em}{}
\titleformat{\subsection}{\normalfont\bfseries}{\thesubsection}{1em}{}
\titleformat{\subsubsection}{\normalfont\bfseries}{\thesubsubsection}{1em}{}

\newtheorem{Theorem}{Theorem}[section]
\newtheorem{Lemma}[Theorem]{Lemma}
\newtheorem{Proposition}[Theorem]{Proposition}

\theoremstyle{definition}

\newtheorem{Remark}[Theorem]{Remark}

\numberwithin{equation}{section}

\definecolor{darkgreen}{rgb}{0,0.5,0}
\definecolor{rem}{rgb}{0.8,0,0}
\definecolor{new}{rgb}{0.3,0.1,0.9}
\definecolor{reply}{rgb}{0,0,0.8}
\definecolor{gray}{gray}{0.7}
\usepackage[
        colorlinks, citecolor=darkgreen,
        backref,
        pdfauthor={Nuno Freitas, Michael Stoll},
        pdftitle={The Generalized Fermat Equation x² + y³ = z²⁵},
]{hyperref}

\usepackage[alphabetic,backrefs,lite]{amsrefs} 

\usepackage{comment}
\usepackage{multirow}

\renewcommand{\C}{\mathbb{C}}
\newcommand{\F}{\mathbb{F}}
\newcommand{\PP}{\mathbb{P}}
\newcommand{\Q}{\mathbb{Q}}

\newcommand{\Z}{\mathbb{Z}}
\newcommand{\rhobar}{{\overline{\rho}}}
\newcommand{\eps}{\varepsilon}
\newcommand{\calI}{\mathcal{I}}
\newcommand{\calO}{\mathcal{O}}
\newcommand{\fp}{\mathfrak{p}}

\DeclareMathOperator{\rank}{rank}
\newcommand{\GL}{\operatorname{GL}}

\renewcommand{\Re}{\operatorname{Re}}
\renewcommand{\Im}{\operatorname{Im}}

\begin{document}

\title {The Generalized Fermat Equation $x^2 + y^3 = z^{25}$}
\author{Nuno Freitas}
\address{Instituto de Ciencias Matemáticas (ICMAT),
         C/ Nicolás Cabrera 13-15
         28049 Madrid, Spain}
\email{nuno.freitas@icmat.es}

\author{Michael Stoll}
\address{Mathematisches Institut,
         Universit\"at Bayreuth,
         95440 Bayreuth, Germany}
\email{Michael.Stoll@uni-bayreuth.de}


\begin {abstract}
  We consider the generalized Fermat equation (*)~$x^2 + y^3 = z^{25}$.
  Using the known parameterization of the primitive integral solutions
  to $x^2 + y^3 = z^5$ (due to Edwards), we reduce the solution of~(*)
  to the solution of five specific equations of the form $H(u,v) = w^5$,
  where $H$ is homogeneous of degree~$10$ with coefficients in a sextic
  number field~$K$, $u$ and~$v$ are coprime (rational) integers, and $w \in K$.
\end {abstract}

\maketitle


\section{Introduction}

The Generalized Fermat Equation
\begin{equation} \label{E:GFE1}
  x^2 + y^3 = z^n,  \quad n \geq 6
\end{equation}
is conjectured to have only finitely many non-trivial primitive solutions $(a,b,c) \in \Z^3$
when the ``Catalan'' solutions $(\pm 3)^2 + (-2)^3 = 1^n$ are counted only once;
here \emph{non-trivial} means $abc \neq 0$ and \emph{primitive} means that $\gcd(a,b,c) = 1$.
(This is a special case of the more general conjecture that there are only finitely
many non-trivial primitive solutions to all equations $x^p + y^q = z^r$ with
$1/p + 1/q + 1/r \le 1$.)

Since any integer $n \ge 6$
is divisible by $6$, $8$, $9$, $10$, $15$, $25$ or a prime $p \ge 7$,
to study equation~\eqref{E:GFE1} it suffices to consider these exponents.
The cases $n=6, 7, 8, 9, 10$ and~$15$ have been fully resolved and the case $n=11$ has been solved assuming GRH.
For a more detailed discussion of~\eqref{E:GFE1},
including a list of the known solutions and references to all solved cases,
we refer the reader to the introduction of~\cite{FNS23n}.

In this paper we focus on the last remaining non-prime case for~$n$, that is, the equation
\begin{equation} \label{eqn:main}
  x^2 + y^3 = z^{25} \,.
\end{equation}

We base our attempt at solving~\eqref{eqn:main}
on the complete parameterization of the coprime solutions to $x^2 + y^3 = z^5$ obtained
by Edwards in~\cite{Ed}. More precisely, using this parameterization, we reduce
the resolution of~\eqref{eqn:main} to the resolution of several related equations
that we proceed to solve using different methods to compute rational points.
Unfortunately, we are left with five~equations that are out of reach of current methods
and computational resources even when assuming~GRH. Our main result is the following.

\begin{Theorem} \label{thm:main}
  If five specific equations of the form
  \[ H(u,v) = w^5 \,, \]
  where $H$ is homogeneous of degree~$10$ with coefficients in a sextic number field~$K$
  (depending on the equation), $u$ and~$v$ are coprime (rational) integers and $w \in K$,
  have only the expected solutions (see Table~\ref{table:summary} below), then
  the only primitive integer solutions to~\eqref{eqn:main}
  are the trivial solutions $(\pm 1, -1, 0)$, $(\pm 1, 0, 1)$, $(0, 1, 1)$, $(0, -1, -1)$
  and the non-trivial Catalan solutions $(\pm 3, -2, 1)$.

  The assumption above can be verified (or shown to be false) if we are able to determine the sets of
  $L$-rational points on five genus~$2$ curves given by an equation of the form
  \[ Y^2 = X^5 + A \]
  (with $A \in L$), where $L$ is a number field of degree~$12$ (depending on the
  equation; it is a quadratic extension of~$K$ above).
\end{Theorem}

After recalling Edwards' parameterization in Section~\ref{sec:Edwards},
which leads to the consideration of $27$~equations
of the form $z^5 = -h_i(u, v)$, where $h_i$ is a binary form of degree~$12$, we
determine in Section~\ref{sec:types} the various possible factorization types of~$h_i$
over $\Q$ and over~$\Q(\sqrt{5})$.
We then eliminate some forms by local considerations; see Section~\ref{sec:local}.
In Section~\ref{sec:Frey} we determine which of the forms in Edwards' list
give rise to points on which twist of~$X(5)$; these twists correspond to either
symplectic or anti-symplectic $5$-congruences with an elliptic curve from
the list of seven curves obtained in~\cite{FNS23n}.

In the remaining Sections \ref{sec:ft1110}, \ref{sec:ft48}, \ref{sec:ft66}
and~\ref{sec:ft12}, we deal with the various factorization types in turn.
In all cases but the last (when $h_i$ is irreducible over~$\Q(\sqrt{5})$), we are able
to determine all solutions to~\eqref{eqn:main} arising from each form~$h_i$
of the corresponding factorization type, using a variety of methods.
Some of the last set of forms can also be excluded, and the remaining ones lead to the
equations mentioned in Theorem~\ref{thm:main} above.

In terms of the classification obtained in Section~\ref{sec:Frey}, our result
is summarized in Table~\ref{table:summary}. We give the corresponding curve from the list
in~\cite{FNS23n} (where ``reducible'' refers to the case of reducible $5$-torsion)
and the symplectic type as ``$+$'' or~``$-$''. In the equations $H_i(u, v) = w^5$
below, $i$ is the number of an Edwards form, $H_i$ is a binary form of degree~$10$
with coefficients in a sextic number field~$K$ depending on~$i$,
$u$ and~$v$ are coprime (rational) integers, and $w \in K$.

\begin{table}[htb]
  \[ \renewcommand{\arraystretch}{1.2}
    \begin{array}{|c|c|c|} \hline
      \text{curve}     & \text{solutions to~\eqref{eqn:main}} & \text{condition} \\ \hline\hline
      \text{reducible} & (\pm 1, -1, 0) & \text{---} \\ \hline
      27a1+            & (\pm 1, 0, 1)  & \text{---} \\ \hline
      54a1-            & \text{---}     & H_{22}(u,v) = w^5 \;\Longrightarrow\; (u, v) = (\pm 1, 0) \\ \hline
      96a1+            & \text{---}     & H_{6}(u,v) = w^5 \;\Longrightarrow\; (u, v) = (0, \pm 1) \\ \hline
      288a1+           & \pm(0, 1, 1)   & \text{---} \\ \hline
      864a1+           & \text{---}     & H_{24}(u,v) = w^5 \;\Longrightarrow\; (u, v) = (\pm 1, 0) \\ \hline
      864a1-           & \text{---}     & \text{---} \\ \hline
      864b1+           & (\pm 3, -2, 1)  & H_{5}(u,v) = w^5 \;\Longrightarrow\; (u, v) = (0, \pm 1) \\ \hline
      864b1-           & \text{---}     & \text{---} \\ \hline
      864c1+           & \text{---}     & H_{16}(u,v) = w^5 \;\Longrightarrow\; (u, v) = (0, \pm 1) \\ \hline
      864c1-           & \text{---}     & \text{---} \\ \hline
    \end{array}
  \]
  \smallskip
  \caption{Summary of the main result.} \label{table:summary}
\end{table}

The number fields~$K$ are given in terms of the minimal polynomial of a generator
in Table~\ref{table:fields}.
How the forms~$H_i$ can be obtained is explained in Section~\ref{sec:ft12}.

\begin{table}[htb]
  \[ \renewcommand{\arraystretch}{1.2}
    \begin{array}{|c|c|} \hline
      i  & \text{minimal polynomial}         \\ \hline \hline
      5  & x^6 + 10 x^3 + 24 x + 5           \\ \hline
      6  & x^6 - 2 x^5 - 6 x - 3             \\ \hline
      16 & x^6 + 10 x^3 - 15 x^2 + 18 x - 10 \\ \hline
      22 & x^6 + 3 x^5 - 10 x^3 + 12 x - 4   \\ \hline
      24 & x^6 - 10 x^3 - 6 x + 5            \\ \hline
    \end{array}
  \]
  \smallskip
  \caption{Coefficient fields of the degree~$10$ forms.} \label{table:fields}
\end{table}

A Magma script verifying the computational claims in the paper
is available on GitHub at~\cite{code}.


\subsection*{Acknowledgments}

We used the Magma Computer Algebra System~\cite{MAGMA} for the computations.

We thank the Max-Planck-Institut für Mathematik in Bonn for its hospitality on several
occasions, which enabled the authors to collaborate on this paper.


\section{The Edwards Parameterization} \label{sec:Edwards}

We base our attempt at solving the equation $x^2 + y^3 = z^{25}$ in coprime integers
on the complete parameterization of the coprime solutions to $x^2 + y^3 = z^5$ obtained
by Edwards in~\cite{Ed}, which we now quickly summarize.

In the following, the notation $h = [\alpha_0, \alpha_1, \dotsc, \alpha_{12}]$
means that $h$ is the binary form 
\[ h(u,v) = \sum_{i=0}^{12} \binom{12}{i} \alpha_i u^{i} v^{12-i} \,. \]
We define binary forms $h_1, \ldots, h_{27}$ as given in Table~\ref{table:forms}.

\begin{table}[htb]
  \noindent\hrulefill

  \smallskip

  \[ \renewcommand{\arraystretch}{1.25}
    \begin{array}{l@{{}={}}l}
      h_1    & [0,1,0,0,0,0,-144/7,0,0,0,0,-20736,0] \\
      h_2    & [-1,0,0,-2,0,0,80/7,0,0,640,0,0,-102400] \\
      h_3    & [-1,0,-1,0,3,0,45/7,0,135,0,-2025,0,-91125] \\
      h_4    & [1,0,-1,0,-3,0,45/7,0,-135,0,-2025,0,91125] \\
      h_5    & [-1,1,1,1,-1,5,-25/7,-35,-65,-215,1025,-7975,-57025] \\
      h_6    & [3,1,-2,0,-4,-4,24/7,16,-80,-48,-928,-2176,27072] \\
      h_7    & [-10,1,4,7,2,5,80/7,-5,-50,-215,-100,-625,-10150] \\
      h_8    & [-19,-5,-8,-2,8,8,80/7,16,64,64,-256,-640,-5632] \\
      h_9    & [-7,-22,-13,-6,-3,-6,-207/7,-54,-63,-54,27,1242,4293] \\
      h_{10} & [-25,0,0,-10,0,0,80/7,0,0,128,0,0,-4096] \\
      h_{11} & [6,-31,-32,-24,-16,-8,-144/7,-64,-128,-192,-256,256,3072] \\
      h_{12} & [-64,-32,-32,-32,-16,8,248/7,64,124,262,374,122,-2353] \\
      h_{13} & [-64,-64,-32,-16,-16,-32,-424/7,-76,-68,-28,134,859,2207] \\
      h_{14} & [-25,-50,-25,-10,-5,-10,-235/7,-50,-49,-34,31,614,1763] \\
      h_{15} & [55,29,-7,-3,-9,-15,-81/7,9,-9,-27,-135,-459,567] \\
      h_{16} & [-81,-27,-27,-27,-9,9,171/7,33,63,141,149,-67,-1657] \\
      h_{17} & [-125,0,-25,0,15,0,45/7,0,27,0,-81,0,-729] \\
      h_{18} & [125,0,-25,0,-15,0,45/7,0,-27,0,-81,0,729] \\
      h_{19} & [-162,-27,0,27,18,9,108/7,15,6,-51,-88,-93,-710] \\
      h_{20} & [0,81,0,0,0,0,-144/7,0,0,0,0,-256,0] \\
      h_{21} & [-185,-12,31,44,27,20,157/7,12,-17,-76,-105,-148,-701] \\
      h_{22} & [100,125,50,15,0,-15,-270/7,-45,-36,-27,-54,-297,-648] \\
      h_{23} & [192,32,-32,0,-16,-8,24/7,8,-20,-6,-58,-68,423] \\
      h_{24} & [-395,-153,-92,-26,24,40,304/7,48,64,64,0,-128,-512] \\
      h_{25} & [-537,-205,-133,-123,-89,-41,45/7,41,71,123,187,205,-57] \\
      h_{26} & [359,141,-1,-21,-33,-39,-207/7,-9,-9,-27,-81,-189,-81] \\
      h_{27} & [295,-17,-55,-25,-25,-5,31/7,-5,-25,-25,-55,-17,295]
    \end{array}
  \]

  \smallskip

  \noindent\hrulefill

  \medskip

  \caption{Definition of the forms $h_i$, $1 \le i \le 27$. \label{table:forms}}
\end{table}

For $i \in \{1, \ldots, 27\}$, let
\[
  g_i = \frac{1}{132^2} \left(\frac{\partial^2{h_i}}{\partial{u}^2} \frac{\partial^2{h_i}}{\partial{v}^2}
         - \frac{\partial^2{h_i}}{\partial{u}\partial{v}} \frac{\partial^2{h_i}}{\partial{u}\partial{v}}\right)
  \quad\text{and}\quad
  f_i = \frac{1}{240} \left(\frac{\partial{h_i}}{\partial{u}} \frac{\partial{g_i}}{\partial{v}}
         - \frac{\partial{h_i}}{\partial{v}} \frac{\partial{g_i}}{\partial{u}}\right) \,.
\]
Note that $f_i$, $g_i$ and~$h_i$ are binary forms with integral coefficients,
of degrees $30$, $20$ and~$12$, respectively.

\begin{Theorem}[Edwards \cite{Ed}*{pages 235--236}] \label{thm:Ed}
  Suppose $a$, $b$, $c$ are coprime rational integers satisfying
  $a^2 + b^3 + c^5 = 0$. Then for some $i \in \{1, \ldots, 27\}$, there
  is a pair of coprime rational integers $u$, $v$ and a choice of sign~$\pm$ such that
  \[ a = \pm f_i(u,v)\,, \qquad b = g_i(u,v)\,, \qquad c = h_i(u,v)\,.  \]
\end{Theorem}

Suppose that we have a primitive solution $(x,y,z)$ to equation \eqref{eqn:main}.
Then Theorem~\ref{thm:Ed} implies that there are coprime integers $u$ and~$v$ satisfying
\begin{equation} \label{eqn:mainsuper}
  C_i \colon z^5 = -h_i(u,v) \,.
\end{equation}
Thus, we have reduced the initial problem to the determination of
the triples $(u,v,z)$, with $u$ and~$v$ coprime, satisfying one of the $27$ possible equations~$C_i$.

The following observation is important and will be used several times.

\begin{Remark} \label{Rkg2}
  We note that by construction, all~$h_i$ are $\GL_2(\C)$-equivalent to the icosahedral
  Klein form $u v (u^{10} + 11 u^5 v^5 - v^{10})$. This implies that when $K$ is a number
  field over which $h_i$ splits off a linear factor, $h_i$ will actually have a factorization
  of the form
  \[ h_i(u, v)
      = \ell_1(u, v) \ell_2(u, v) (A \ell_1(u, v)^{10} + B \ell_1(u, v)^5 \ell_2(u, v)^5 + C \ell_2(u, v)^{10}) \,,
  \]
  where $\ell_1$ and~$\ell_2$ are linear forms over~$K$ and $A, B, C \in K$ (with $(B/11)^2 = A C$).
  We will make use of this in the following way. Since $u$ and~$v$ are coprime,
  if $-h_i(u, v) = z^5$, we will have that
  \[ A \ell_1(u, v)^{10} + B \ell_1(u, v)^5 \ell_2(u, v)^5 + C \ell_2(u, v)^{10} = \alpha w^5 \]
  for some $\alpha, w \in K$, where $\alpha$ is in a finite set (that has to
  be determined first). We can write this as either
  \[ \Bigl(2 A \Bigl(\frac{\ell_1(u,v)}{\ell_2(u,v)}\Bigr)^5 + B\Bigr)^2
      = 4 \alpha A \Bigl(\frac{w}{\ell_2(u,v)^2}\Bigr)^5 + B^2 - 4AC \]
  or
  \[ \Bigl(2 C \Bigl(\frac{\ell_2(u,v)}{\ell_1(u,v)}\Bigr)^5 + B\Bigr)^2
      = 4 \alpha C \Bigl(\frac{w}{\ell_1(u,v)^2}\Bigr)^5 + B^2 - 4AC \,. \]
  So a solution gives rise to $K$-rational points on two genus~$2$ curves over~$K$,
  \[ Y^2 = 4 \alpha A X^5 + (B^2 - 4AC) \qquad\text{and}\qquad Y^2 = 4 \alpha C X^5 + (B^2 - 4AC) \, \]
  where the points are obtained from
  \[ X = \frac{w}{\ell_2(u,v)^2}\,, \quad
     Y = 2 A \Bigl(\frac{\ell_1(u,v)}{\ell_2(u,v)}\Bigr)^5 + B \]
  for the first curve, and similarly (with $(\ell_1, A)$ swapped with $(\ell_2, C)$) for the second curve.
  (One can scale the $X$ and~$Y$ coordinates to obtain a curve of the form $Y^2 = X^5 + \gamma$
  if desired.)

  Note that both curves will be conjugate over the field over which the product $\ell_1 \ell_2$
  can be defined when $\ell_1$ and~$\ell_2$ are not yet defined over the base field. So in this
  case they will provide the same information.
\end{Remark}


\section{Factorization Types} \label{sec:types}

Let $G \in \Q[u,v]$ be a binary form, and let $K$ be a number field. We say $G$ has
\emph{factorization type $[d_1, d_2, \dotsc, d_n]$ over~$K$} if it factors as a product
$G = G_1 G_2 \cdots G_n$, where the binary forms $G_j \in K[u,v]$ are irreducible over~$K$ of degree~$d_j$.
Table~\ref{table:factTypes} records the
factorization types of~$h_i$ over~$\Q$ or~$\Q(\sqrt{5})$ for $i \in \{1, 2, \ldots, 27\} \setminus \{7, 11, 19\}$.
Note that we will show in the next section that the forms associated to $i \in \{7, 11, 19\}$
(which are irreducible over~$\Q(\sqrt{5})$) cannot lead to primitive solutions of~\eqref{eqn:main}.

\begin{table}[htb]
  \[ \renewcommand{\arraystretch}{1.2}
    \begin{array}{|c|c|} \hline
      \text{factorization type of $h_i$}  & i \in I \\ \hline\hline
      [ 1, 1, 10 ] \text{\ over\ } \Q & 1, 20, 25 \\ \hline
      [ 4, 8 ] \text{\ over\ } \Q & 3, 4, 12, 17, 18, 27\\ \hline
      [ 6, 6 ] \text{\ over\ } \Q(\sqrt{5}) & 2, 10, 26 \\ \hline
      [12] \text{\ over\ } \Q(\sqrt{5}) & 5, 6, 8, 9, 13, 14, 15, 16, 21, 22, 23, 24 \\ \hline
    \end{array}
  \]
  \smallskip
  \caption{Factorization types}
  \label{table:factTypes}
\end{table}

In the remainder of this work we will study the solutions of~\eqref{eqn:mainsuper}
by applying a strategy that is adapted to each factorization type.


\section{Local Primitive Solutions} \label{sec:local}

We can check for which pairs $(u, v)$ of coprime $p$-adic integers equation~$C_i$ above
has a solution in~$\Z_p$ that comes from a primitive $p$-adic solution of Equation~\eqref{eqn:main}.
In particular, we find the following.

\begin{Lemma} \label{lem:local}
  When $i \in \{7, 11, 19\}$, a solution to equation~\ref{eqn:main} in coprime $2$-adic integers
  cannot give rise to a solution of equation~$C_i$ in $2$-adic integers.
\end{Lemma}

\begin{proof}
  This is an easy calculation: for two pairs of coprime residue classes mod~$2$,
  the value of $-h_i(u, v)$ has $2$-adic valuation~$1$ and therefore cannot be a
  fifth power, and for the third coprime residue class, the values of $f_i$, $g_i$ and~$h_i$
  are all even, so we do not obtain a primitive solution.
\end{proof}

By the same kind of computation, we obtain restrictions on the
location of the point $(u : v) \in \PP^1(\Q_p)$ for $p = 2$ and~$3$.
These conditions are recorded in Table~\ref{table:resclasses}.

\begin{table}[htb]
  \[ \renewcommand{\arraystretch}{1.2}
    \begin{array}{|l|c|c|}
      \hline
          &  \multicolumn{2}{c|}{\text{non-excluded residue classes of $(u,v)$}} \\ \hline
       i  & p = 2 & p = 3 \\ \hline\hline
       1  & (8u, 1) & (81u, 1) \\ \hline
       2  & (u, 1)  & (3u, 1), (3u+2, 1), (1, 3v) \\ \hline
       3  & (2u, 1), (1, 2v) & (u,1) \\ \hline
       4  & (2u, 1), (1, 2v) & (u,1) \\ \hline
       5  & (2u, 1), (1, 2v) & (3u, 1), (3u+2, 1), (1, 3v) \\ \hline
       6  & (u, 1) & (1, 81v+51) \\ \hline
       8  & (u, 1) & (3u, 1), (3u+2, 1), (1, 3v) \\ \hline
       9  & (2u, 1), (1, 2v) & (u,1) \\ \hline
      10  & (u, 1) & (3u, 1), (3u+1, 1), (1, 3v) \\ \hline
      12  & (1, v) & (3u+2, 1), (1, 3v), (3u,1) \\ \hline
      13  & (1, v) & (3u+1,1), (3u,1), (1,3v) \\ \hline
      14  & (2u, 1), (1, 2v) & (3u+1,1), (3u,1), (1,3v) \\ \hline
      15  & (2u, 1), (1, 2v) & (u,1) \\ \hline
      16  & (2u, 1), (1, 2v) & (3u+2,1), (3u+1,1), (1,3v) \\ \hline
      17  & (2u, 1), (1, 2v) & (u,1) \\ \hline
      18  & (2u, 1), (1, 2v) & (u,1) \\ \hline
      20  & (8u, 1) & (1, 81v) \\ \hline
      21  & (2u, 1), (1, 2v) & (3u+2,1), (3u,1), (1,3v) \\ \hline
      22  & (1, 8v+6) & (u,1) \\ \hline
      23  & (1,v) & (1, 81v + 66) \\ \hline
      24  & (u, 1) & (3u+1,1), (3u,1), (1,3v) \\ \hline
      25  & (2u, 1), (1, 2v) & (81u + 80, 1) \\ \hline
      26  & (2u, 1), (1, 2v) & (u,1) \\ \hline
      27  & (2u, 1), (1, 2v) & (3u+2, 1), (3u,1), (1, 3v) \\ \hline
    \end{array}
  \]
  \smallskip
  \caption{Residue classes of solutions.}
  \label{table:resclasses}
\end{table}


\section{The Frey curve and $X_E^\pm(5)$} \label{sec:Frey}

As explained in~\cite{FNS23n, PSS}, to a primitive solution~$(a,b,c)$ of the generalized Fermat equation
\begin{equation}\label{eq:fermat}
 x^2 + y^3 = z^n,
\end{equation}
where $n \geq 5$ is an integer,
we can attach the Frey curve
\[E := E_{(a,b,c)} \; : \; y^2 = x^3 + 3bx - 2a.\]
We write $\rhobar_{E,p}$ to denote the mod~$p$ Galois representation attached to~$E$.
The following lemma improves the conclusion of~\cite[Proposition~2.4]{FNS23n} under
additional $2$-adic and $3$-adic hypotheses.

\begin{Lemma}\label{lem:irred}
  Let $(a,b,c)$ be a primitive solution to~\eqref{eq:fermat} with exponent $n \geq 7$.
  Assume at least one of the following conditions:
  \begin{enumerate}[\upshape(i)]
    \item $a$ is even or $b \not\equiv 0, -1, 4 \bmod 8$;
    \item $a \not\equiv \pm 1 \bmod 9$ or $b \not\equiv -1 \bmod 3$.
  \end{enumerate}
  Then for all primes~$p \geq 5$, the representation~$\rhobar_{E,p}$ is absolutely irreducible.
\end{Lemma}

\begin{proof}
  Observe that if $3 \mid c$, then $3^2 \mid 3^n \mid a^2 + b^3$. Therefore, $a \not\equiv \pm 2, \pm 4 \bmod 9$
  if \hbox{$b \equiv -1 \bmod 3$} as otherwise $a^2 + b^3$ would be divisible by~$3$ but not by~$9$.
  It follows that $(a \bmod 9, b \bmod 3)$ must satisfy one of the congruence conditions
  in~\cite[Table~2]{FNS23n}. From the proof of this table in \emph{loc.~cit.}, we conclude that the Frey curve
  $E = E_{a,b,c}$ is isomorphic over~$\Q_3$ to a quadratic twist of the curve(s) listed in the fifth column
  of the same row of the table. Note that except for the curve~$96a1$, all the curves~$W$ in the table
  have conductor $N_W$ satisfying $v_3(N_W) = 3$, hence they have semistability defect $e = 12$ and
  a supercuspidal (hence irreducible) inertial type at~$3$ (see~\cite[Table~1]{DFV}).
  Since $p \geq 5$ does not divide $e = 12$, we conclude that the inertial type remains irreducible
  after reduction, that is, $\rhobar_{E,p}|_{I_3}$ is irreducible (here $I_3 \subset G_\Q$
  is an inertia subgroup at~$3$). So $\rhobar_{E,p}$ is irreducible if we avoid~$96a1$, that is,
  when $a$ and~$b$ satisfy the $3$-adic conditions in the statement.

  We now apply a similar argument using the Frey curve~$E/\Q_2$ using~\cite[Table~1]{FNS23n}.
  Indeed, except for the curve~$54a1$ or the case that $E$ has a quadratic twist with good reduction
  (curve~$27a1$ or the last row of the table\footnote{Note that \cite[Table~1]{FNS23n} says this case is impossible;
  however that is proved with a global argument using irreducibility of~$\rhobar_{E,p}$ which is
  what we are trying to prove. For our local argument the last row of the table cannot be excluded.})
  all the curves~$W$ in the table have semistability defect $e = 8$ and $v_2(N_W) = 5$. As above,
  this yields an irreducible inertial type after reduction. We conclude that $\rhobar_{E,p}$
  is irreducible if $a$ and~$b$ satisfy the $2$-adic conditions in the statement.

  Our assumptions on $(a,b)$ let us avoid at least one of the two bad cases above,
  thus $\rhobar_{E,p}$ is irreducible, hence absolutely irreducible since $\Q$ is totally real.
\end{proof}

From Table~\ref{table:factTypes} we see that for
\[ i \in \calI = \left\{ 5, 6, 8, 9, 13, 14, 15, 16, 21, 22, 23, 24 \right\} \]
the corresponding forms are irreducible over~$\Q$ and also over~$\Q(\sqrt{5})$.

Let now $(a,b,c)$ be a primitive solution to~\eqref{eqn:main}
arising from a triple~$(\pm f_i,g_i, h_i)$ with $i \in \mathcal{I}$, and write $E$ for the associated Frey curve.

For each $i \in \mathcal{I}$ and for all $u, v \in \Z/8\Z$, not both even, we compute the pairs
\[ (a, b) = \bigl(\pm f_i(u,v) \bmod 4, g_i(u,v) \bmod 8\bigr) \]
and verify in which row (indexed by~$i_2$) of~\cite[Table~1]{FNS23n} it falls.
Similarly, for each $i \in \mathcal{I}$ and for all $u, v \in \Z/9\Z$, not both divisible by~$3$,
we compute the pairs
\[ (a, b) = \bigl(\pm f_i(u,v) \bmod 9, g_i(u,v) \bmod 3\bigr) \]
and verify in which row (indexed by~$i_3$) of~\cite[Table~2]{FNS23n} it falls.
Combining this information with \cite[Table~3]{FNS23n}, we can associate to each $i \in \calI$
an elliptic curve~$W_i$ with Cremona label in $\{54a1, 96a1, 864a1, 864b1, 864c1\}$;
this correspondence can be found in Table~\ref{table:itoW}. The twists listed in Table~\ref{table:itoW}
come from intersecting the twists in~\cite[Tables~1 and~2]{FNS23n}; note that, for each pair~$(a,b)$
as above, the required twist~$d$ may differ.

Moreover, the previous calculations also show that the resulting pairs~$(a,b)$ always satisfy
the hypotheses of Lemma~\ref{lem:irred}, hence $\rhobar_{E,5}$ is absolutely irreducible.
We now apply~\cite[Lemma~2.3]{FNS23n} to~$E[5]$ and conclude that
there is a quadratic twist~$E^{(d)}$ of~$E$ for some $d \in \{\pm 1, \pm 2, \pm 3. \pm 6\}$
such that the the $5$-torsion module~$E^{(d)}[5]$ is isomorphic to~$W_i[5]$ as given by Table~\ref{table:itoW}.
Note that \cite[Lemma~2.3]{FNS23n} as stated applies only to the $p$-torsion
representations~$\rhobar_{E_{(a,b,c)},p}$ with $p \geq 7$ a prime, but we can also apply it in our setting.
Indeed, we observe that
\begin{enumerate}[(i)]
  \item the conductor calculations in its proof work for any integer exponent $n \geq 7$
        (as used in the proof of Lemma~\ref{lem:irred} above) so, in particular, it holds for our exponent $n=25$;
  \item if $\rhobar_{E,5}$ is absolutely irreducible, then the level lowering and twisting parts
        of the argument carry through in the same way.
\end{enumerate}

Since $\rhobar_{E,5}$ has non-abelian image, by~\cite{FKsymplectic}*{Theorem~2.1}, the symplectic type
of the isomorphism $E^{(d)}[5] \simeq W_i[5]$ of $G_\Q$-modules is well defined. That is,
either all such isomorphisms are symplectic or all are anti-symplectic (i.e., either all raise the
Weil pairing to a power with square exponent or all raise it to a non-square exponent).
Thus we obtain a rational point on $X_{W_i}^+(5)$ or~$X_{W_i}^-(5)$, respectively, but not on both.
To decide which of these is the case for each~$i$, we apply two symplectic criteria.

For $i= 22$ we have $W_{22} = 54a1$ and both $W_{22}$ and the Frey curve $E = E^{(d)}_{(a,b,c)}$
have multiplicative reduction at~$2$, so we can apply~\cite{FKsymplectic}*{Theorem~1.20}.
The minimal discriminants of $E$ and~$W_{22}$ are, respectively, $\Delta(E) = 2^{-6} 3^3 d^6 c^{25}$
and $\Delta(W_{22}) = -2^3 3^9$; thus $v_2(\Delta(E)) \equiv -2 v_2(\Delta(W_{22})) \bmod 5$,
and since $-2$ is a non-square mod~$5$, we conclude from~\cite{FKsymplectic}*{Theorem~1.20} that
the isomorphism $E[5] \simeq W_{22}[5]$ is anti-symplectic. Now observe that, for all $i \neq 22$, the
curve~$W_i$ has semistability defect $e = 8$ and conductor~$N_i$ with $2$-adic valuation $v_2(N_i) = 5$,
so we can apply~\cite{FKsymplectic}*{Theorem~1.13(A)}.
To this end we used a Magma implementation of this theorem.
The results can be found in the last column of Table~\ref{table:itoW}.

For completeness we include also in Table~\ref{table:itoWCM} the forms that give rise
to primitive solutions whose associated Frey curve $E = E_{(a,b,c)}^{(d)}$ is mod~5 congruent
to the two CM curves in the statement of~\cite{FNS23n}*{Lemma~2.3}, namely $27a1$ and~$288a1$.
As above, the irreducibility of~$\rhobar_{E,5}$ follows from the inertial types at $2$ or~$3$,
and the level lowering and twisting arguments are exactly the same.
From \cite{FNS23n}*{Theorem~5.1} we know that we only have to consider the `+' type in all cases;
note that although this theorem is stated for $p \geq 11$, the part of its proof required here
applies for $p=5$ since we know that $\rhobar_{E,5}$ is irreducible.

\begin{table}[htb]
  \[ \renewcommand{\arraystretch}{1.2}
    \begin{array}{|c|c|c|c|} \hline
      i  &   W   &       d      & \text{type} \\ \hline\hline
      5  & 864b1 & \pm 2, \pm 6 & + \\ \hline
      6  &  96a1 & \pm 3        & + \\ \hline
      8  & 864c1 & \pm 2, \pm 6 & - \\ \hline
      9  & 864b1 & \pm 2, \pm 6 & - \\ \hline
      13 & 864b1 & \pm 1, \pm 3 & + \\ \hline
      14 & 864c1 & \pm 1, \pm 3 & + \\ \hline
    \end{array} \qquad
    \begin{array}{|c|c|c|c|} \hline
      i  &   W   &       d      & \text{type} \\ \hline\hline
      15 & 864a1 & \pm 2, \pm 6 & - \\ \hline
      16 & 864c1 & \pm 2, \pm 6 & + \\ \hline
      21 & 864a1 & \pm 1, \pm 3 & - \\ \hline
      22 &  54a1 &     1,    -3 & - \\ \hline
      23 &  96a1 & \pm 6        & + \\ \hline
      24 & 864a1 & \pm 2, \pm 6 & + \\ \hline
    \end{array}
  \]
  \smallskip
  \caption{Correspondence for $i \in \calI$}
  \label{table:itoW}
\end{table}

\begin{table}[htb]
  \[ \renewcommand{\arraystretch}{1.2}
    \begin{array}[t]{|c|c|c|c|} \hline
      i  &   W    &       d      &  \text{type} \\ \hline\hline
      2  & 27a1   & \pm 2, \pm 6 & + \\ \hline
      3  & 288a1  & \pm 1, \pm 3 & + \\ \hline
      4  & 288a1  & \pm 2, \pm 6 & + \\ \hline
      10 & 27a1   & \pm 2, \pm 6 & + \\ \hline
      12 & 288a1  & \pm 2, \pm 6 & +   \\ \hline
    \end{array} \qquad
    \begin{array}[t]{|c|c|c|c|} \hline
      i  &   W   &       d      & \text{type} \\ \hline\hline
      17 & 288a1 & \pm 1, \pm 3 & + \\ \hline
      18 & 288a1 & \pm 2, \pm 6 & + \\ \hline
      26 & 27a1  & \pm 2, \pm 6 & + \\ \hline
      27 & 288a1 & \pm 2, \pm 6 & + \\ \hline
    \end{array}
  \]
  \smallskip
  \caption{Correspondence to CM curves}
  \label{table:itoWCM}
\end{table}


\section{Factorization type $[1,1,10]$ over $\Q$} \label{sec:ft1110}

Here $i \in \{1, 20, 25\}$ and the polynomial $h_i$ factors over~$\Q$ as the product of two linear factors
and a factor of degree~$10$. This is a case where the associated mod~$5$ representations are reducible.

We apply Remark~\ref{Rkg2} over~$\Q$ to the three relevant forms. We have
\begin{align*}
  h_1(u,v)    &= -12 u v (12^4 u^{10} + 11 \cdot 12^2 u^5 v^5 - v^{10}) \\
  h_{20}(u,v) &= -12 u v (2^8 u^{10} + 11 \cdot 12^2 u^5 v^5 - 3^4 v^{10}) \\
  h_{25}(u,v) &= -3 (u+v) (u-v) \bigl((u+v)^{10} + 11 \cdot 3^2 (u+v)^5 (u-v)^5 - 3^4 (u-v)^{10}\bigr)
\end{align*}
Using the coprimality of $u$ and~$v$ and the congruence conditions from Table~\ref{table:resclasses},
we find that $\alpha = 12, 12, 1$, respectively (where we take $(\ell_1, \ell_2) = (u, v)$ or $(u+v, u-v)$).
The recipe of Remark~\ref{Rkg2} then gives the following pairs of genus~$2$ curves (after scaling
the coordinates).
\[ \begin{array}{@{i = {}}l@{\,: \qquad Y^2 = X^5 + {}}l@{\qquad\text{and}\qquad Y^2 = X^5 + {}}l}
    1 & 2^8 \cdot 5^3           & 2^6 \cdot 3^4 \cdot 5^3 \\
   20 & 2^8 \cdot 3^4 \cdot 5^3 & 2^6 \cdot 5^3 \\
   25 & 2^8 \cdot 5^3           & 2^8 \cdot 3^6 \cdot 5^3
\end{array} \]
Using the Magma function \texttt{MordellWeilGroup}, we determine that the group of rational
points on the Jacobian variety of the first curve, $Y^2 = X^5 + 2^8 \cdot 5^3$, is infinite
cyclic. Then Magma's \texttt{Chabauty} function, applied to a generator of this group,
shows that the only rational points on this curve are the point at infinity and a pair
of points with $x$-coordinate~$-4$.

For~$h_1$, these points correspond to $(u,v) = (0,1)$ and~$(0,1)$
(one point does not lift), of which the first gives the trivial solution $(-1, -1, 0)$ and
the second does not give rise to a primitive solution to the original equation.

For~$h_{25}$, these points correspond to $(u,v) = (1,1)$ and~$(1,-1)$
(again, one point does not lift), both of which do not lead to a primitive solution.

Similarly, we find that the group of rational points on the Jacobian of the curve $Y^2 = X^5 + 2^6 \cdot 5^3$
is trivial, which implies that the point at infinity is the only rational point on that curve.
It corresponds to taking $(u,v) = (1,0)$ in~$h_{20}$, which does not give a primitive solution.

So the only primitive solutions of Equation~\eqref{eqn:main} that arise from this case
are the trivial solutions $(\pm 1, -1, 0)$.


\section{Factorization type $[4,8]$ over $\Q$} \label{sec:ft48}

The mod~$5$ Galois representation of the associated Frey curves is isomorphic to that of $288a1$.
Here $i \in \{3, 4, 12, 17, 18, 27\}$ and the polynomial $h_i$ factors over~$\Q$ as
\[ h_i(u,v) = h_{i,4}(u,v) h_{i,8}(u,v), \]
where $h_{i,4}$ and $h_{i,8}$ are factors in $\Z[u,v]$ of degree $4$ and $8$, respectively.
With a computer we easily verify the following.
\begin{enumerate}[(1)]
  \item The prime divisors of the resultant of $h_{i,4}$ and~$h_{i,8}$ are $2,3,5$.
        Thus, the prime divisors of $\gcd(h_{i,4}(u,v), h_{i,8}(u,v))$ belong to $\{2,3,5\}$
        when $u$ and~$v$ are coprime integers.
  \item For $p = 2$ and $p = 3$ we see that $p$ does not divide~$h_{i,4}(u,v)$ when $u$ and~$v$
        satisfy the conditions in Table~\ref{table:resclasses}.
  \item Similarly, $p = 5$ never divides~$h_{i,8}(u,v)$ when $(u,v)$ gives rise to a $5$-adically
        primitive solution.
\end{enumerate}
This implies that for coprime $u, v \in \Z$ coming from a primitive solution of~\eqref{eqn:main}
we always have
\[ \gcd\bigl(h_{i,4}(u,v), h_{i,8}(u,v)\bigr) = 1 \,. \]
It now follows that, if $(u, v, z_0)$ satisfies~\eqref{eqn:mainsuper} for~$h_i$, then there is
an integer~$z$ (not divisible by $2$ or~$3$) such that
\begin{equation} \label{eqn:twocurves}
  z^5 = h_{i,4}(u,v) \,.
\end{equation}

We will now show that solutions of~\eqref{eqn:twocurves} give rise to rational points on certain
hyperelliptic curves. It can be easily checked that, over $\Q(\sqrt{-1})$, we have a factorization
\[ h_{i,4}(u,v) = H_i(u,v) \bar{H}_i(u,v) \]
into conjugate degree two polynomials (see Table~\ref{table:datahyper} for explicit expressions
for~$H_i$).
Moreover, the resultant of $H_i$ and~$\bar{H}_i$ is divisible only by primes above $2$ and ~$3$.
Since we already know that $z$ above is not divisible by $2$ and~$3$,
this implies that $H_i(u,v)$ (and $\bar{H}_i(u,v)$) must be a fifth power in~$\Z[\sqrt{-1}]$
(note that all units in~$\Z[\sqrt{-1}]$ are fifth powers.)
So there are $a, b \in \Z$ such that
\begin{equation} \label{eqn:gaussintegers}
  H_i(u,v) = (a + b \sqrt{-1})^5 = (a^5 - 10 a^3 b^2 + 5 a b^4) + (5a^4 b - 10 a^2 b^3 + b^5) \sqrt{-1} \,.
\end{equation}
Comparing coefficients, this gives
\begin{equation} \label{eqn:realimaginary}
  a^5 - 10 a^3 b^2 + 5 a b^4 = \Re(H_i(u,v)) \,, \qquad
  5a^4 b - 10 a^2 b^3 + b^5  = \Im(H_i(u,v)) \,,
\end{equation}
where $\Re$ and~$\Im$ denote the real and imaginary parts.
Now, for each $i$, we find integers $\gamma_i$, $\alpha_i$, $\beta_i$ such that
\begin{equation} \label{eqn:prehyper}
  \gamma_i \bigl(\alpha_i \Re(H_i(u,v)) + \beta_i \Im(H_i)(u,v)\bigr) \Im(H_i(u,v))
    = S_i(u,v)^2
\end{equation}
for some polynomial $S \in \Z[u,v]$. Using~\eqref{eqn:realimaginary} in~\eqref{eqn:prehyper}
and writing
\begin{equation}  \label{eqn:change}
  X = \frac{b}{a} \qquad\text{and}\qquad Y = \frac{S_i(u,v)}{a^5} \,,
\end{equation}
we obtain hyperelliptic curves
\[ M_i \colon Y^2 = F_i(X) \,, \]
where $F_i$ is of degree~$10$ and factors as
\[ F_i(X) = X (X^4 - 10 X^2 + 5) G_i(X) \]
with a polynomial~$G_i$ of degree~$5$. Each solution~$(u,v)$ of~\eqref{eqn:mainsuper}
(for one of the indices~$i$ considered here) that comes from a primitive solution
of~\eqref{eqn:main} then gives rise to a rational point on the corresponding curve~$M_i$.

Up to isomorphism, there are only two curves~$M_i$ (which are quadratic twists by~$-1$
of each other):
\begin{align*}
  F_4(X) &= 3 X (X^4 - 10 X^2 + 5) (X^5 + 10 X^4 - 10 X^3 - 20 X^2 + 5 X + 2) \\
  F_{18}(X) &= F_4(-X) \\
  F_{12}(X) = F_{17}(X) &= -F_4(X) \\
  F_3(X) = F_{27}(X) &= -F_4(-X)
\end{align*}

A summary of the information involved for all $i$ can be found in Table~\ref{table:datahyper}.

\begin{table}[htb]
  \[ \renewcommand{\arraystretch}{1.25}
    \begin{array}{|r|c|c|c|c|c|} \hline
      i &           H_i                         & \gamma_i   & \alpha_i   & \beta_i   & S_i(u,v)        \\ \hline\hline
      3 & (3 + 6\sqrt{-1}) u^2 + v^2            &    3       &    2       &     -1    & 6 u v           \\ \hline
      4 & (-3 + 6\sqrt{-1}) u^2 + v^2           &    3       &    2       &      1    & 6 u v           \\ \hline
      12 & (2 - 3\sqrt{-1}) u^2 + 2 u v + 2 v^2 &   -3       &    2       &      1    & 3 u (u + 2v)    \\ \hline
      17 & 3 u^2 + (1 - 2\sqrt{-1}) v^2         &   -3       &    2       &      1    & 6 u v           \\ \hline
      18 & 3 u^2 + (-1 - 2\sqrt{-1}) v^2        &    3       &   -2       &      1    & 6 u v           \\ \hline
      27 & (2 + \sqrt{-1}) u^2 + (2 - 2\sqrt{-1}) u v + (2 + \sqrt{-1})v^2 & -3 & -2 & 1 & 3(u^2 - v^2) \\ \hline
    \end{array}
  \]
  \smallskip
  \caption{Data for the hyperelliptic curves $M_i$.}
  \label{table:datahyper}
\end{table}

Note that $(0, 0) \in M_i(\Q)$ is a rational point for all~$i$.
Applying a partial descent over~$\Q$ as in~\cite{SiksekS}
to $M_3 \cong M_{12} \cong M_{17} \cong M_{27}$, we obtain a Selmer set with only
one element~$\xi_i$, which must correspond to the point~$(0,0)$.
Therefore, $\xi_i$ is of the form $(?,5,?)$. Thus, given a rational point in~$M_i(\Q)$, there is
a rational point with the same $X$-coordinate on the genus 1 curve given by
\[ \mathcal{C} \colon 5 Y^2 = X^4 - 10 X^2 + 5 \,. \]
The Jacobian of~$\mathcal{C}$ is the elliptic curve
\[ E \colon y^2 = x^3 + x^2 - 83 x + 88 \,, \]
which satisfies $E(\Q) \cong \Z/2\Z$. This implies that
\[ \mathcal{C}(\Q) = \{(0, 1), (0, -1)\} \,, \]
which implies in turn that $(0, 0)$ is the only rational point on~$M_i$.

This argument does not work for $M_4 \cong M_{18}$: $M_4$ has in addition rational
points with $(1, \pm 12)$ (on~$M_{18}$, they are $(-1, \pm 12)$), and the Selmer
set correspondingly contains another element $(?, -1, ?)$. Unfortunately,
the Jacobian elliptic curve of
\[ Y^2 = -(X^4 - 10 X^2 + 5) \]
has positive rank, so we cannot conclude in the same way.
On the other hand, a $2$-descent on
the Jacobian of~$M_4$ shows that the Mordell-Weil rank is at most~$1$,
and the difference $P = [(1, 12) - (0, 0)]$ has infinite order, so the rank is
exactly~$1$ and $P$ generates a subgroup of finite index. We can therefore
use Chabauty's method to determine $M_4(\Q)$; see below. The result is that
\[ M_4(\Q) = \{(0, 0), (1, \pm 12)\} \qquad \text{and so} \qquad M_{18}(\Q) = \{(0, 0), (-1, \pm 12)\} \,. \]

Consider a rational point $(b/a, S_i(u,v)/a^5) \in M_i(\Q)$. From~\eqref{eqn:realimaginary},
the expression for~$H_i$ and the fact that $u$ and~$v$ are coprime, we deduce
that $a$ and~$b$ are coprime (no nontrivial fifth power of an integer
can divide both $\Re(H_i(u,v))$ and~$\Im(H_i(u,v)))$).

We now determine the primitive solutions to equation~\eqref{eqn:main}
arising from $(0, 0) \in M_i(\Q)$. Then $b = 0$ and~$a = \pm 1$, so
\[ \Re(H_i(u, v)) = \pm 1 \qquad \text{and} \qquad \Im(H_i(u, v)) = 0 \,. \]
For $i \in \{12, 17, 17, 27\}$, this pair of equations has no integral solution,
and for $i \in \{3, 4\}$, the only solutions are $(u,v) = (0, \pm 1)$ or $(\pm 1, 0)$.
These give rise to the trivial solutions $(0,1,1)$, $(0,-1,-1)$ and some further non-primitive
solutions with $a = 0$ to equation~\eqref{eqn:main}.

We now consider the points $(1, \pm 12) \in M_4(\Q)$. By the above, we have $a = b = \pm 1$.
Using this in~\eqref{eqn:realimaginary}, we obtain the impossible equation $\pm 4 = 6 u^2$.

Similarly, the points $(-1, \pm 12) \in M_{18}(\Q)$ lead to $a = -b = \pm 1$
and thence to the impossible equation $\pm 4 = -2 v^2$.

So the only primitive solutions of~\eqref{eqn:main} arising from this factorization pattern
are~$\pm(0, 1, 1)$.

\medskip

It remains to carry out the Chabauty argument for~$M_4$.

We use the odd degree model
\[ C \colon y^2 = 30 x^9 + 75 x^8 - 360 x^7 - 300 x^6 + 756 x^5 + 330 x^4 - 360 x^3 - 60 x^2 + 30 x + 3 \]
of~$M_4$, which is obtained by replacing~$x$ with $1/x$; our goal is then to show
that the rational points of~$C$ are the point at infinity and the two points~$(1, \pm 12)$.
Let $J$ denote the Jacobian variety of~$C$.

We work over~$\Q_{11}$ since $J(\F_{11}) \cong \Z/2\Z \times \Z/2\Z \times \Z/70\Z \times \Z/70\Z$
has relatively small exponent. Let $P = [(1,12) - \infty] \in J(\Q)$ be the point of infinite order
mentioned above. Then $70 P$ is in the kernel of
reduction mod~$11$. By a standard computation, the details of which we omit here as similar
computations can be found in the literature (see also the Magma code at~\cite{code}),
we determine that the image in~$\Omega^1(C/\F_{11})$ of the space of differentials
in~$\Omega^1(J/\Q_{11})$ that annihilate~$P$ is generated by
\[ \frac{x - 3}{2y}\,dx\,, \qquad \frac{x^2 - 1}{2y}\, dx \qquad\text{and}\qquad \frac{x^3}{2y}\,dx \,. \]
The last one of these does not vanish at the point at infinity or at the points $(1, \pm 1)$
in~$C(\F_{11})$. From the computation of the $2$-Selmer set of~$C$ we know that each
rational point on~$C$ must differ from $\infty$ or~$(1, 12)$ by twice an element of~$J(\Q)$.
Checking this condition mod~$11$ shows that this implies that every rational point on~$C$
reduces mod~$11$ to either~$\infty$ or a point with $x$-coordinate~$1$.
The non-vanishing of the reduction of some annihilating differential on these residue
classes implies that there can only be one rational point in each of the residue classes,
which tells us that there cannot be further rational points.


\section{Factorization type $[6,6]$ over $\Q(\sqrt{5})$} \label{sec:ft66}

The mod~$5$ Galois representation of the associated Frey curves is isomorphic to that of~$27a1$.
Here $i \in \{2, 10, 26\}$. The argument in this section is similar to the one used in section~\ref{sec:ft48}.

We observe that
\begin{align*}
  -h_2(-u/2, v) &= -h_{10}(v/2, u) = -h_{26}((u+v)/2,(u-v)/2) \\
                &= 25 u^{12} + 275 u^9 v^3 - 165 u^6 v^6 - 55 u^3 v^9 + v^{12} \\
                &=: h(u,v) \,.
\end{align*}
A solution $(u,v)$ in coprime integers to any of the three relevant equations will
result in a solution $(-2u, v)$, $(2v, u)$ or $(u+v, u-v)$ in coprime integers of $z^5 = h(u,v)$
(taking into account the restrictions modulo~$2$ from Table~\ref{table:resclasses}).

The polynomial $h(u,v)$ is a quartic in $u^3$ and~$v^3$ that factors over~$\Q(\sqrt{5})$
into two conjugate quadratics in $u^3$ and~$v^3$:
\[ h(u,v) = \Bigl(v^6 - \frac{55 + 27\sqrt{5}}{2} u^3 v^3 - 5 u^6\Bigr)
              \Bigl(v^6 - \frac{55 - 27\sqrt{5}}{2} u^3 v^3 - 5 u^6\Bigr) \,.
\]
The resultant of the two factors is $-3^{18} \sqrt{5}^{12}$, so for coprime integers
$u$ and~$v$, the gcd of the two factors on the right hand side is of the form $3^e \sqrt{5}^{e'}$,
which implies that each factor is a unit times $3^e \sqrt{5}^{e'}$ times a fifth power.
We can assume that $e, e' \in \{0, 1, \ldots, 4\}$. Since the product, which is
$3^{2e} 5^{e'}$ times the fifth power of a rational integer, must be a fifth power,
it follows that $e = e' = 0$. So there is some $\alpha \in \Z[\eps]$, where $\eps = (1 + \sqrt{5})/2$,
and some $j \in \{-2, -1, 0, 1, 2\}$ such that
\[ \eps^j \alpha^5 = \Bigl(v^6 - \frac{55 + 27\sqrt{5}}{2} u^3 v^3 - 5 u^6\Bigr) \,. \]
Writing $\alpha = a + b \sqrt{5}$, we see that
\[ \alpha^5 = (a^5 + 50 a^3 b^2 + 125 a b^4) + (5a^4 b + 50 a^2 b^3 + 25 b^5) \sqrt{5} \,. \]
Comparing coefficients, we obtain for $j = 0$ that
\[ a^5 + 50 a^3 b^2 + 125 a b^4 = v^6 - \tfrac{55}{2} u^3 v^3 - 5 u^6 \quad\text{and}\quad
   5a^4 b + 50 a^2 b^3 + 25 b^5 = \tfrac{27}{2} u^3 v^3 \,.
\]
Taking appropriate $\Q(\sqrt{-5})$-linear combinations, we then have
\begin{align*}
  (a^5 + 50 a^3 b^2 + 125 a b^4) - \frac{55 \pm 4\sqrt{-5}}{27} (5a^4 b + 50 a^2 b^3 + 25 b^5)
    = (v^3 \pm \sqrt{-5} u^3)^2 \,.
\end{align*}
We take the product of these two (and multiply by~$9^2$) to finally obtain
\begin{align*}
  (9 (v^6 + 5 u^6))^2
    &= 81 a^{10} - 1650 a^9 b + 16725 a^8 b^2 - 99000 a^7 b^3 + 395250 a^6 b^4 - 1039500 a^5 b^5 \\
    &\quad{} + 1961250 a^4 b^6 - 2475000 a^3 b^7 + 2128125 a^2 b^8 - 1031250 a b^9 + 215625 b^{10} \\
    &=: F(a, b) \,.
\end{align*}
This shows that a solution will lead to a rational point on the hyperelliptic curve of genus~$4$
given by $D_0 \colon y^2 = F(x, 1)$, with $x = a/b$ and $y = 9(v^6 + 5 u^6)/b^5$.

When $j \ne 0$, we have to expand instead $\eps^j (a + b \sqrt{5})^5 = g_1(a, b) + g_2(a, b) \sqrt{5}$
with homogeneous polynomials $g_1, g_2 \in \Z[a,b]$ of degree~$5$ and then perform a similar
computation to arrive at another hyperelliptic curve~$D_j$ of genus~$4$.

We will see below that $D_2(\Q) = D_1(\Q) = \emptyset$, that $D_0(\Q)$ consists of the
two points at infinity, and that $D_{-1}(\Q)$ and~$D_{-2}(\Q)$ consist of the two points
with $x$-coordinate~$1$ (and $y = \pm 192$ and $\pm 96$, respectively).

Note that the resultant of the two quadratic forms
$V^2 - \tfrac{55 \pm 27\sqrt{5}}{2} U V - 5 U^2$ is $-3^6 \sqrt{5}^4$, which implies
(since $3^5$ is the maximal fifth power of a rational integer that can divide their~$\gcd$
when evaluated at the coprime integers $u^3$ and~$v^3$; note that if an integer~$n$ divides
$a$ and~$b$, then $n$ divides both $\alpha$ and its conjugate, so $n^5$ divides both factors)
that either $a, b \in \Z$ and $\gcd(a,b) \mid 3$ or $a, b \in \tfrac{1}{2} + \Z$ and
$\gcd(2a, 2b) \mid 3$.
So we either have $b = 0$ (from the point at infinity of~$D_0$)
and $a \in \{\pm 1, \pm 3\}$
or $a = b \in \{\pm\tfrac{1}{2}, \pm\tfrac{3}{2}, \pm 1, \pm 3\}$ (from the points with $x = 1$
on $D_{-1}$ and~$D_{-2}$). Taking into account
that $u$ and~$v$ are integers, this gives that
\[ v^6 + 5 u^6 \in \{1, 3^5, 3^4, 2 \cdot 3^4, 2^5 \cdot 3^4, 2^6 \cdot 3^4\} \,. \]
The only solutions in coprime integers are $(u, v) = (0, \pm 1)$; they correspond to
the trivial solution~$(1,0,1)$.

So the only primitive solutions to~\eqref{eqn:main} arising from this case are the
trivial solutions~$(\pm 1, 0, 1)$.

We write $J_j$ for the Jacobian of~$D_j$. Using Magma, we can easily check the following facts
about the curves~$D_j$.
\begin{enumerate}[(1)]
  \item For $j = 1, 2$ the routine \texttt{TwoCoverDescent} returns empty fake Selmer sets.
        Thus $D_1(\Q)$ and $D_2(\Q)$ are both empty.
  \item A small search using \texttt{RationalPoints} finds the points~$(1 : \pm 9 : 0)$ on~$D_0$,
        $(1, \pm 192)$ on~$D_{-1}$ and~$(1, \pm 96)$ on~$D_{-2}$.
  \item The routine \texttt{TorsionBound} tells us that $J_j(\Q)$ is torsion free for $j = 0,-1,-2$.
  \item The routine \texttt{RankBound} gives a bound of~$2$ for the rank of the Mordell-Weil groups
        of $J_{-1}$ and~$J_{-2}$. By taking the difference of the two known points, we get a point
        in the Jacobian of infinite order. Thus $1 \leq \rank(J_j(\Q)) \leq 2$ for $j = -1, -2$.
  \item The routine \texttt{RankBound} gives a bound of~$4$ for the rank of the Mordell-Weil group
        of $J_0$. Again, taking the difference of the two known points we get a point of infinite order.
        Thus $1 \leq \rank(J_0(\Q)) \leq 4$.
\end{enumerate}

In Theorem~9.1 of~\cite{StollSelmer}, it is shown that $D_0(\Q) = \{(1 : \pm 9 : 0)\}$
using ``Elliptic curve Selmer group Chabauty'' (assuming~GRH for the necessary class group computations).

It remains to show that $D_{-1}$ and~$D_{-2}$ have only the points mentioned earlier.
These curves are hyperelliptic of genus~$4$ and the Mordell-Weil rank of their Jacobians is
at most~$2$. Thus we can hope to find all rational points by applying Chabauty's method.
This can be done provided that we can find generators of a subgroup of finite index of~$J_j(\Q)$.

We first write down the following nicer models for $D_j$ (which are obtained as
\[ \text{\texttt{SimplifiedModel(ReducedMinimalWeierstrassModel($D_j$))}} \]
in Magma).
\begin{align*}
  \tilde{D}_{-1} \colon \quad Y^2  &= 36X^{10} - 120X^9 + 705X^8 - 60X^7 + 3060X^6 + 3846X^5 \\
                                   &\quad{} + 6390X^4 + 5340X^3 + 3345X^2 + 1230X + 189 \\
  \tilde{D}_{-2} \colon \quad Y^2  &= 9X^{10} - 30X^9 + 645X^8 + 1860X^7 + 6390X^6 + 11274X^5 + 15660X^4 \\
                                   &\quad{} + 14460X^3 + 8805X^2 + 3120X + 516
\end{align*}
We write $\tilde{J}_j$ for their Jacobians.
A small search for rational points reveals only the points at infinity on both curves.
These points were expected to be found since they are the images of the points computed
in the previous section under the isomorphism between the models.
For both curves we will denote by $\infty_{+}$ and~$\infty_{-}$ the points at infinity with
positive and negative $y$-coordinate, respectively. The objective of the remainder of this section
is to show that these two points are indeed all the points in~$\tilde{D}_j(\Q)$.

\begin{Lemma} Let $P \in \tilde{D}_j(\Q)$. Then the divisor class $[P - \infty_{-}]$ is in~$2\tilde{J}_j(\Q)$.
  \label{lem:doubleJ}
\end{Lemma}

\begin{proof}
  Applying the routine \texttt{TwoCoverDescent} to~$\tilde{D}_j$ returns a \textit{fake 2-Selmer set}
  with only one element. By construction we have a factorization of the degree~$10$ polynomials
  defining~$\tilde{D}_j$ into two conjugate polynomials over~$\Q(\sqrt{-5})$. This implies that the
  fake 2-Selmer set equals the 2-Selmer set. Let $C$ be the unique 2-covering of~$\tilde{D}_j$ that has
  points everywhere locally. Any rational point $P \in \tilde{D}_j(\Q)$ lifts to a rational point on
  some covering having points everywhere locally. Since $C$ is unique, both $\infty_{-}$ and~$P$
  lift to a rational point on~$C$. This implies that $[P - \infty_{-}] \in 2\tilde{J}_j(\Q)$.
\end{proof}

A computer search reveals the following points $Q_{j,1}, Q_{j,2} \in \tilde{J}_j(\Q)$ given in Mumford notation.
\begin{align*}
  Q_{-1,1} &= (x^4 + 2x^3 + 4x^2 + 3x + 1, \; -60x^3 - 90x^2 - 90x - 30) \\
  Q_{-1,2} &= (x^4 - \tfrac{53}{27}x^3 + \tfrac{4}{9}x + \tfrac{1}{9}, \;
               -\tfrac{52693}{243}x^3 + \tfrac{448}{3}x^2 + \tfrac{8063}{81}x + \tfrac{1118}{81}) \\
  Q_{-2,1} &= (x^4 + 2x^3 + 4x^2 + 3x + 1, \; -30x^3 - 45x^2 - 45x - 15) \\
  Q_{-2,2} &= (x^4 + \tfrac{1}{5}x^3  + \tfrac{23}{5}x^2  + \tfrac{21}{5}x + \tfrac{3}{5}, \;
               -\tfrac{683}{25}x^3  + \tfrac{321}{25}x^2  + \tfrac{1982}{25}x + \tfrac{666}{25})
\end{align*}
We note also that $2Q_{j,1} = [\infty_{+} - \infty_{-}]$.

\begin{Lemma}
  The group $\tilde{J}_j(\Q)$ is torsion-free and of rank~$2$. The points $Q_{j,1}$ and~$Q_{j,2}$ generate
  a subgroup $G_j \subseteq \tilde{J}_j(\Q)$ of finite odd index.
\end{Lemma}

\begin{proof}
  We already know that that $\tilde{J}_j(\Q)$ is torsion-free and of rank at most~$2$.
  We check that in both cases, the image of the subgroup generated by $Q_{j,1}$ and~$Q_{j,2}$
  under $\tilde{J}_j(\Q) \to \tilde{J}_j(\F_{29}) \to \tilde{J}_j(\F_{29})/2\tilde{J}_j(\F_{29})$
  is isomorphic to~$(\Z/2\Z)^2$. This implies the second claim.
\end{proof}

We now want to apply the Chabauty-Coleman method to prove that $\infty_{\pm}$ are the only
rational points on~$\tilde{D}_j$. We will use the prime $p = 7$. Write $\rho_7$ for the mod~$7$
reduction maps $\tilde{J}_j(\Q) \to \tilde{J}_j(\F_7)$ and $\tilde{D}_j(\Q) \to \tilde{D}_j(\F_7)$.
Denote by $\iota$ the Abel-Jacobi embeddings
\[ \iota \colon \tilde{D}_j \to \tilde{J}_j \quad \text{given by} \quad \iota(P) = [P - \infty_{-}] \,. \]

\begin{Proposition} \label{prop:pts7}
  Let $G_j = \langle Q_{j,1}, Q_{j,2} \rangle$ as above. Then $\rho_7(G_j) = \rho_7(\tilde{J}_j(\Q))$. Furthermore,
  \begin{enumerate}[\upshape(1)]
    \item If $P \in \tilde{D}_{-1}(\Q)$, then $\rho_7(P) \in \{(1 : 1 : 0), (1 : 6 : 0), (0 : 0 : 1), (1 : 0 : 1)\}$;
    \item if $P \in \tilde{D}_{-2}(\Q)$, then $\rho_7(P) \in \{(1 : 3 : 0), (1 : 4 : 0) \}$.
  \end{enumerate}
\end{Proposition}

\begin{proof}
  The group~$\rho_7(G_j)$ has index~$2$ in~$\tilde{J}_j(\F_7)$. Since $G_j$ has odd index in~$\tilde{J}_j(\Q)$,
  we conclude that $\rho_7(G_j) = \rho_7(\tilde{J}_j(\Q))$.

  Now for the second statement. Let $P \in \tilde{D}_j(\Q)$. From Lemma~\ref{lem:doubleJ} we know that
  $\iota(P) \in 2 \tilde{J}_j(\Q)$. Thus, we also have $\rho_7(\iota(P)) \in \rho_7(2\tilde{J}_j(\Q))$.
  With the computer we easily check that the points $P \in \tilde{D}_j(\F_7)$ such that
  $[P - \rho_7(\infty_{-})] \in \rho_7(2\tilde{J}_j(\Q)) = 2 \rho_7(G_j)$ are those listed in the statement.
\end{proof}

The embedding~$\iota$ induces an isomorphism~$\iota^*$ from the space of regular
1-forms~$\Omega(\tilde{J}_j/\Q_7)$
to~$\Omega(\tilde{D}_j/\Q_7)$. This isomorphism is independent of the base-point~$\infty_{-}$ in
the definition of~$\iota$. Furthermore, there is a well defined integration pairing
\[ \Omega(\tilde{D}_j/\Q_7) \times \tilde{J}_j(\Q_7) \to \Q_7\,, \qquad (\omega, D) \mapsto \int_0^D \iota^{*-1} \omega \,. \]
We need to compute a basis for the differentials that annihilate~$\tilde{J}_j(\Q) \subset \tilde{J}_j(\Q_7)$ under this pairing.
Since the genus of $\tilde{D}_j$ is~$4$ and the Mordell-Weil rank of~$\tilde{J}_j(\Q)$ is~$2$,
we will find two differentials in such a basis. We will now sketch how to obtain this basis of differentials.

We first find two independent points $R_{j,1}$, $R_{j,2}$ in the intersection of~$\tilde{J}_j(\Q)$ and the
kernel of reduction mod~$7$. We do this because integrals on the kernel of reduction can be computed
via power series. We then compute the integrals $\int_0^{R_{j,k}} \iota^{*-1}\omega$ for $\omega$
running through a basis of~$\Omega(\tilde{D}_j/\Q)$ to sufficient $7$-adic precision and finally
find the two basis differentials by linear algebra. (We omit the details, as similar
computations are described elsewhere in the literature.)

For $j = -1$, the differentials in the basis reduce mod~$7$ to
\begin{equation}
  \frac{x^3 + 5 x^2}{2y}\,dx \qquad\text{and}\qquad \frac{x + 4}{2y}\,dx \,,
\end{equation}
and for $j = -2$ the reductions mod~$7$ are
\begin{equation}
  \frac{x^2}{2y}\,dx \qquad\text{and}\qquad \frac{x + 4}{2y}\,dx \,.
\end{equation}

We now determine the rational points on~$\tilde{D}_{-1}$. We see that the differentials mod~$7$
above do not both vanish at the points in part~(1) of Proposition~\ref{prop:pts7}. This implies that
there is at most one point in~$\tilde{D}_{-1}(\Q)$ reducing to each of these points. In particular,
the points~$\infty_{\pm}$ are the only rational points in their residue classes.
Now let $P \in \tilde{D}_{-1}(\F_7)$ be $(0 : 0 : 1)$ or~$(1 : 0 : 1)$. Suppose there is a rational point
$W = (x : y : z)$ reducing to~$P$ such that $y \neq 0$. Then $(x : -y : z)$ is another rational point
reducing to~$P$. But this is not possible since there can be at most one rational point in these
residue classes. Since there is no rational point with vanishing $y$-coordinate in these residue classes,
we conclude that there cannot be any other rational points.

We now determine the rational points on~$\tilde{D}_{-2}$. We consider the points in part~(2) of Proposition~\ref{prop:pts7}.
They are the images of the two rational points at infinity. The first of the two differentials mod~$7$
given above vanishes there to first order. This implies that there are at most two rational points
on~$\tilde{D}_{-2}$ belonging to each of these two residue classes. To show that there is in fact only
one rational point, we take a closer look at the power series expansion of the integral of the second differential
at a point at infinity in terms of a uniformizer~$t$ there. We obtain
\[ I(t) = (13 \cdot 7 + O(7^3)) t - (19 \cdot 7 + O(7^3)) t^2 - (103 + O(7^3)) t^3 + O(t^4) \,. \]
Let $0 \neq a \in \Z_7$. Then (taking into account that the coefficient of~$t^m$ in~$I(t)$
has valuation at least $-v_7(m)$) it follows that
\begin{equation} \label{eqn:expeta2}
  \frac{I(7a)}{7^2 a} = 13 + O(7^2) - (19 \cdot 7 + O(7^3)) a - (103 \cdot 7 + O(7^4)) a^2 + O(7^2)
\end{equation}
which is always $\equiv 6 \bmod 7$, so can never vanish. This shows that $t = 0$ is the only
zero of~$I(t)$ in~$7 \Z_7$, and so the point at infinity must be the only rational point
in its residue class.

\medskip

This concludes the proof of
\[ D_{-1}(\Q) = \{(1, -192), (1, 192)\} \qquad\text{and}\qquad
   D_{-2}(\Q) = \{(1, -96), (1, 96)\}
\]
(for the original models).


\section{Forms that are irreducible over $\Q(\sqrt{5})$} \label{sec:ft12}

The remaining forms we have to deal with are those which are not only irreducible 
over~$\Q$ but also over~$\Q(\sqrt{5})$, that is $i \in \calI = \{5, 6, 8, 9, 13, 14, 15, 16, 21, 22, 23, 24\}$.

In all cases, there is a sextic number field~$K_i$ (that only depends on the mod~$5$
representation of the associated Frey curves) such that $h_i$ splits as a quadratic form
times a degree~$10$ form~$H_i$ with coefficients in the ring of integers~$\calO_{K_i}$ of~$K_i$;
we choose the scaling so that $H_i$ is primitive (i.e., its coefficients
generate the unit ideal of~$\calO_{K_i}$).

The resultant of the two factors is divisible only by primes above $2$, $3$ and~$5$
with only one prime above~$5$ occurring. For $i \notin \{6, 22, 23\}$, we know that
$c$ is not divisible by $2$ or~$3$, where $c^5 = -h_i(u,v)$ for coprime integers $u$ and~$v$,
so $H_i(u,v)$ cannot be divisible by primes above $2$ and~$3$. For $i \in \{6, 22, 23\}$,
we can show (using the congruence conditions from Table~\ref{table:resclasses}) that
any prime above $2$ or~$3$ must show up with an exponent that is a multiple of~$5$.
The same holds for the possible prime above~$5$ (which in fact cannot show up).
Since all the relevant sextic fields have trivial class group, this implies that
a primitive solution to our original equation~\eqref{eqn:main} that comes from one of
the form considered here must lead to integers $u$ and~$v$ such that $H_i(u,v)$ equals
a unit times a fifth power in~$\calO_{K_i}$. We note that this property is invariant
under simultaneous scaling of $u$ and~$v$ by a nonzero factor (as that results in
multiplying $H_i(u,v)$ by the tenth power of the scaling factor).

We now use a sieving procedure to reduce the number of units (which come from a
set of representatives of $\calO_{K_i}^\times$ modulo fifth powers; the fields~$K_i$
have signature $(2,2)$, so the unit rank is~$3$, and we have a priori $5^3 = 125$
different units to consider). If $p > 5$ is a prime, we can run through representatives
of all nonzero pairs $(u,v) \in \F_p^2$ modulo scaling and check which units are
compatible with the resulting value
$H_i(u,v) \in \calO_{K_i}/p\calO_{K_i} \cong \prod_{\fp \mid p} \calO_{K_i}/\fp$.
We can also run a similar computation working modulo~$25$ (which actually leads
to quite strong restrictions).

The result of this computation is that \emph{all} units can be excluded when
$i \in \{8, 9, 15, 21\}$ (these are the forms that correspond to \emph{anti}symplectic
congruences with one of $854a1$, $864b1$ or~$964c1$),
and in all other cases, only one unit remains. We replace
$H_i$ by $H_i$ divided by the corresponding unit, so that the relevant equation
is now
\begin{equation} \label{eqn:deg10}
  H_i(u, v) = w^5 \qquad\text{with $u, v \in \Z$ and $w \in K_i$.}
\end{equation}

When we write below ``\dots gives the Frey curve~$E$'', this means that $E$
can be obtained as a quadratic twist of the Frey curve associated to the given
pair~$(u,v)$ and index~$i$. Except for the form with $i = 5$, the solutions
of~\eqref{eqn:deg10} listed below do not give rise to primitive solutions
of the original equation~\eqref{eqn:main}.
\begin{enumerate}[$\bullet$]
  \item $i = 22$: $(u,v) = (1,0)$ gives the Frey curve $54a2$ and leads to
        a solution of~\eqref{eqn:deg10}.
  \item $i = 6, 23$: $(u,v) = (0,1)$ gives the Frey curve $96a1$ and leads
        to a solution of~\eqref{eqn:deg10} (the equations are equivalent).
  \item $i = 24$: $(u,v) = (1,0)$ gives the Frey curve $864a1$ and leads
        to a solution of~\eqref{eqn:deg10}.
  \item $i = 5, 13$: $(u,v) = (0,1)$ gives the Frey curve $864b1$ and leads
        to a solution of~\eqref{eqn:deg10} (the equations are equivalent).
        This gives rise to the primitive pair of solutions $(\pm 3, -2, 1)$
        of~\eqref{eqn:main}.
  \item $i = 14, 16$: $(u,v) = (1,-1)$ (for $i = 14$) and $(u,v) = (0,1)$ (for $i = 16$)
        lead to the Frey curve $864c1$ and to a solution of~\eqref{eqn:deg10}
        (the equations are equivalent).
\end{enumerate}

\begin{Proposition}
  If the solutions of~\eqref{eqn:deg10} listed above (for $i \in \{5, 6, 13, 14, 16, 22, 23, 24\}$)
  are the only ones, then the Catalan solutions $(\pm 3, -2, 1)$ are the only primitive
  solutions to the original equation~\eqref{eqn:main} that arise from forms that
  are irreducible over~$\Q(\sqrt{5})$.
\end{Proposition}

\begin{proof}
  Under the assumption made, all primitive solutions to the original equation arising
  from one of the remaining eight forms of degree~$12$ must have $(u,v)$ as in the
  list above (up to a common sign change). The only primitive solution that is obtained
  when evaluating the triples of Edwards forms at these values is the Catalan solution.
\end{proof}

Taking into account the fact that three pairs among the eight relevant equations~\eqref{eqn:deg10}
are equivalent, and so only five equations need to be considered,
this proves the first statement in Theorem~\ref{thm:main} (given that
we have already dealt with all other factorization patterns).

This raises the question how one could try to establish the assumption in the proposition
above. Note that equation~\eqref{eqn:deg10} really describes a curve that is given
by six equations of the form
\[ (\text{degree~10 in $u$, $v$}) = (\text{quintic in $w_1, \ldots, w_6$}) \,, \]
which are obtained by fixing an integral basis of~$\calO_{K_i}$, writing
$w$ as a linear combination of this basis with coefficients $w_1, \ldots, w_6$
and then comparing coefficients with respect to this basis on both sides.
This is a curve of \emph{very} large genus.

However, one could try to work with the curve over~$K_i$ (which by Riemann-Hurwitz
has genus~$16$) and use the extra condition that $u, v \in \Z$ at a suitable stage.
But this still seems to be infeasible.

An alternative is to make a further quadratic extension from $K_i$ to~$L_i$ such that
the quadratic factor in the original factorization of~$h_i$ over~$K_i$ splits over~$L_i$.
Then we can use Remark~\ref{Rkg2} to obtain a genus~$2$ curve over~$L_i$ of the
form $Y^2 = X^5 + A$. This proves the second statement in Theorem~\ref{thm:main}.
But if we want to use standard methods to try and determine
the $L_i$-rational points on it, we first need to compute its $2$-Selmer group,
which requires information on the class and unit groups of the field~$L_i(\sqrt[5]{A})$
of degree~$60$ over~$\Q$. Unfortunately, this is quite a bit too large for current
technology to give a result in reasonable time, even assuming~GRH.


\end{document}